\documentclass[letterpaper]{article}
\usepackage[T1]{fontenc}

\usepackage[letterpaper,margin=1in]{geometry}
\usepackage[utf8]{inputenc}
\usepackage{times}

\usepackage{enumitem}
\usepackage[table, dvipsnames]{xcolor}
\definecolor{darkgreen}{rgb}{0.0, 0.8, 0.0}

\definecolor{linkcol}{rgb}{0,0,0.38}
\definecolor{citecol}{rgb}{0,0.2,0}
\definecolor{urlcol}{rgb}{0.1,0.35,0}

\usepackage[pdfpagelabels,bookmarks=false,hyperfootnotes=false]{hyperref}
\hypersetup{colorlinks, linkcolor=linkcol, citecolor=citecol, urlcolor=urlcol}

\usepackage{url}

\usepackage{amsmath}
\usepackage{amssymb}
\usepackage{amsthm}
\usepackage{thmtools}
\usepackage{thm-restate}
\usepackage{mathtools}
\usepackage{bbm}
\usepackage{xfrac}

\usepackage[nameinlink, capitalise]{cleveref}

\usepackage{aligned-overset}
\usepackage{tikz}
\usetikzlibrary{calc}
\usetikzlibrary{math}
\usetikzlibrary{shapes.geometric}
\usetikzlibrary{arrows}
\usetikzlibrary{decorations.pathreplacing, decorations.pathmorphing}

\usepackage{float}
\usepackage{graphicx, dsfont}
\usepackage{xspace}
\usepackage[linesnumbered, ruled, vlined]{algorithm2e}
\usepackage[margin=10pt,font=small,labelfont=bf,skip=-5pt]{caption}

\usepackage[textsize=footnotesize, color=orange!40]{todonotes}
\newcommand\bN{\ensuremath{\mathbb{N}}}

\newcommand\cB{\ensuremath{\mathcal{B}}}
\newcommand\cC{\ensuremath{\mathcal{C}}}

\newcommand\cF{\ensuremath{\mathcal{F}}}

\newcommand\cN{\ensuremath{\mathcal{N}}}

\newcommand\cP{\ensuremath{\mathcal{P}}}
\newcommand\cW{\ensuremath{\mathcal{W}}}

\newcommand\vf{\text{VF}}

\newcommand\fodd{\mathcal{F}_\text{odd}}
\newcommand\clouds{\frak{W}}

\renewcommand{\epsilon}{\varepsilon}

\newtheorem{theorem}{Theorem}
\newtheorem{lemma}[theorem]{Lemma}
\newtheorem{corollary}[theorem]{Corollary}
\newtheorem{proposition}[theorem]{Proposition}
\newtheorem{definition}[theorem]{Definition}

\date{}
\title{Improved Erd\H{o}s-P\'osa inequalities for odd cycles in planar graphs}
\author{
    Luise Puhlmann\thanks{Research Institute for Discrete Mathematics and Hausdorff Center for Mathematics, University of Bonn}
    \and
    Niklas Schlomberg\footnotemark[1]
}

\begin{document}

\maketitle

\begin{abstract}
In an undirected graph, the odd cycle packing number is the maximum number of pairwise vertex-disjoint odd cycles.
The odd cycle transversal number is the minimum number of vertices that hit every odd cycle.
The maximum ratio between transversal and packing number is called Erd\H{o}s-P\'osa ratio.
We show that in planar graphs, this ratio does not exceed 4.
This improves on the previously best known bound of 6 by Kr\'al\textquoteright{}, Sereni and Stacho.
\end{abstract}

\section{Introduction}

Given a ground set $V$ and a family $\cC$ of subsets of $V$, a \emph{packing} for $\cC$ is a collection of pairwise disjoint elements of $\cC$.
A \emph{transversal} for $\cC$ is a subset $T \subseteq V$ that intersects every element of $\cC$.
Generally, given a packing $\cP \subseteq \cC$, a transversal for $\cC$ must contain a distinct element of every element of $\cP$.
Thus, the \emph{transversal number} $\tau$, i.e.\@ the minimum cardinality of a transversal for $\cC$, is at least as large as the \emph{packing number} $\nu$, the maximum cardinality of a packing for $\cC$.

A long line of research addresses the question under which circumstances the transversal number can be bounded in terms of the packing number.
Erd\H{o}s and P\'osa~\cite{ErdP65} showed that if $\cC$ is the family of vertex sets of cycles in an undirected graph $G$, there exists a function $f \colon \bN \to \bN$ such that $\tau \leq f(\nu)$.
This property is known as Erd\H{o}s-P\'osa property.
The Erd\H{o}s-P\'osa  property also holds for directed cycles in digraphs~\cite{ReeRST96}.
However, for odd cycles Reed gave a counterexample for the Erd\H{o}s-P\'osa  property, even for graphs embedded in the projective plane: The so-called ``Escher walls'' have packing number 1, but allow for arbitrarily large transversal numbers~\cite{Ree99}.

In planar graphs, odd cycles do have the Erd\H{o}s-P\'osa  property~\cite{Ree99}.
In fact, in this case the function $f$ can even be chosen linear, i.e.\@ the ratio between $\tau$ and $\nu$ can be bounded by a constant.
The smallest such constant is called Erd\H{o}s-P\'osa ratio.
Fiorini et al.\@~\cite{FioHRV07} gave the first constant upper bound on this ratio of $10$, which was improved to $6$ by Kr\'al\textquoteright{}, Sereni and Stacho~\cite{KraSS12}.
The best known lower bound for this ratio is $2$, which is attained on a complete graph on $4$ vertices.
In this paper we combine methods of~\cite{KraSS12} with a new structural lemma to improve the upper bound to $4$:

\begin{restatable}{theorem}{mainthm}\label{thm:main_thm}
Let $G$ be an undirected planar graph.
Then $\tau \leq 4 \nu$, where $\tau$ is the odd cycle transversal number and $\nu$ is the odd cycle packing number.
\end{restatable}
As all of our methods are constructive, a corresponding odd cycle packing and odd cycle transversal can even be computed in polynomial time.
For the problem of finding a maximum-cardinality odd cycle packing in planar graphs however, there also exists a $(3+\varepsilon)$-approximation algorithm for any $\varepsilon > 0$~\cite{SchTV23}.

For the minimum-cardinality odd cycle transversal problem in planar graphs the best known polynomial-time approximation guarantee is $2.4$~\cite{BerY12}.
Both of these results hold not only for odd cycles, but for packing and transversal problems in a rich class of cycle families in planar graphs, which are called \emph{uncrossable} (due to Goemans and Williamson~\cite{GoeW98}).
It turns out that all uncrossable cycle families allow for a finite Erd\H{o}s-P\'osa ratio.
The best upper bound in this general setting is $8.38$~\cite{Sch24}; for the family of arbitrary cycles in $G$, there exists a better upper bound of $3$~\cite{CheFS12}\cite{MaYZ13}.

There are also some results on the Erd\H{o}s-P\'osa property for odd cycles in graphs with other restrictions:
For example, if $G$ is highly connected, the Erd\H{o}s-P\'osa ratio does not exceed $2$~\cite{RauR01}\cite{Tho01}.
Also, the Erd\H{o}s-P\'osa property holds if $G$ can be embedded in an orientable surface of bounded genus, which was first proved by~\cite{KawN07}.
Conforti et al.\@~\cite{ConFHJS20} gave a bound on the Erd\H{o}s-P\'osa ratio with linear dependence on the genus.

There is also an edge-variant of the problem, where $\cC$ is the collection of edge sets of odd cycles in an undirected graph.
Here, the Erd\H{o}s-P\'osa property still does not hold for odd cycles in general graphs, but in planar graphs the Erd\H{o}s-P\'osa ratio is 2~\cite{KraV04}.

\medskip

Our paper is structured as follows:
We first define some useful notation and concepts in \cref{sec:preliminaries}.
\cref{sec:thm} is devoted to the proof of our main theorem (\cref{thm:main_thm}), an outline of our main techniques can be found in \cref{sec:outline}.
The key technical lemmata are proven in \cref{sec:cloud_lemma}.

\section{Preliminaries}
\label{sec:preliminaries}
We start by formally defining odd cycle transversals for general graphs:

\begin{definition}[Odd cycle transversal]
    \label{def:odd_cycle_transversal}
    Given a graph $G = (V,E)$, an odd cycle transversal is a set of vertices $T\subseteq V$ such that $T\cap V(C)\neq \emptyset$ for all odd cycles $C$ in $G$. The transversal number $\tau(G)$ is the size of an odd cycle transversal of minimum size.
\end{definition}

For planar graphs we can give a different definition which is very useful for us.
If $G$ is a planar graph, embedded in the plane, let $\cF(G)$ be the set of faces of $G$.
For a face $F \in \cF(G)$ let $V(F)$ be the set of vertices on the boundary of $F$.
The edge set $E(F)$ is the set of boundary edges of $F$, excluding bridges.
Therefore, $E(F)$ is Eulerian.
We call $F$ \emph{odd} if $|E(F)|$ is odd.
In that case $E(F)$ contains an odd cycle.
We define $\fodd(G) \subseteq \cF(G)$ to be the set of odd faces of $G$.

\begin{proposition}\label{prop:odd_cycle_odd_faces}
Let $G$ be a planar graph, embedded in the plane, and $C$ a cycle in $G$.
Let $\cF(C) \subseteq \cF(G)$ be the set of faces of $G$ inside $C$ (w.r.t.\@ the planar embedding of $G$).
$C$ is odd if and only if $\cF(C)$ contains an odd number of odd faces.
\end{proposition}
\begin{proof}
This follows directly from the fact that $E(C)$ can be written as the symmetric difference 
\[E(C) = \bigoplus_{F \in \cF(C)} E(F).
\qedhere\]
\end{proof}

As in~\cite{KraSS12}, we define the \emph{vertex-face incidence graph}:
\begin{definition}[Vertex-face incidence graph]
	\label{def:vf}
	Given a planar graph $G$ with a fixed planar embedding, its \emph{vertex-face incidence graph} $\vf(G)$ is the planar graph on the vertex set $\cF(G) \cup V(G)$ with the edge set $E(\vf(G))$ being $\{\{F, v\} : F \in \cF(G), v \in V(F)\}$.
	We embed $\vf(G)$ such that each edge $\{v, F\}$ is embedded inside $F$, see \Cref{fig:vf} for an example.
\end{definition}

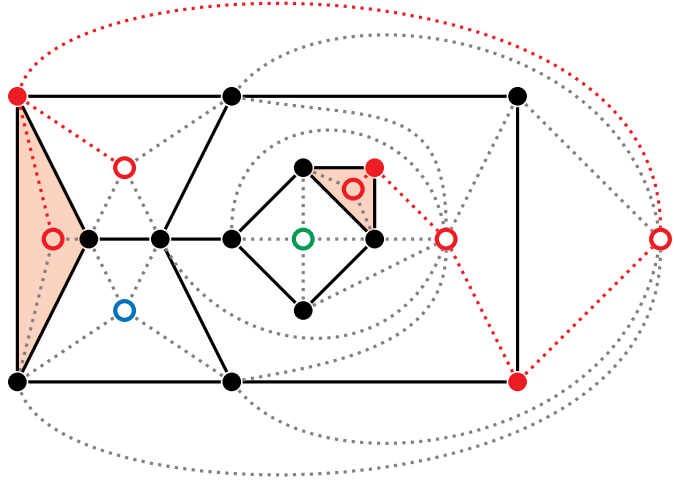
\begin{figure}[h]
	\centering
	  \begin{tikzpicture}[scale=0.95, very thick]
  	\tikzstyle{gvertex} = [fill, circle, inner sep=2.5, color=black]
  	\tikzstyle{vfvertex} = [draw, ultra thick, circle, inner sep=2.5]
  	
  	\fill[Red!20] (0,0) -- (0,4) -- (1,2) -- cycle;
  	\fill[Red!20] (5,2) -- (5,3) -- (4,3) -- cycle;
  	
  	\tikzstyle{vfedge} = [dotted, color=gray]
	\node[gvertex] (a) at (0,0) {};
	\node[gvertex] (b) at (3,0) {};
	\node[gvertex, Red] (c) at (7,0) {};
	\node[gvertex] (d) at (7,4) {};
	\node[gvertex] (e) at (3,4) {};
	\node[gvertex, Red] (f) at (0,4) {};
	\node[gvertex] (g) at (1,2) {};
	\node[gvertex] (h) at (2,2) {};
	\node[gvertex] (i) at (3,2) {};
	\node[gvertex] (j) at (4,1) {};
	\node[gvertex] (k) at (5,2) {};
	\node[gvertex, Red] (l) at (5,3) {};
	\node[gvertex] (m) at (4,3) {};
	\draw (a) -- (b) -- (c) -- (d) -- (e) -- (f) -- (a) -- (g) -- (h) -- (b);
	\draw (f) -- (g);
	\draw (h) -- (e);
	\draw (h) -- (i) -- (j) -- (k) -- (l) -- (m) -- (i);
	\draw (k) -- (m);
	
	\node[vfvertex, Red] (A) at (0.5, 2) {};
	\node[vfvertex, RoyalBlue] (B) at (1.5, 1) {};
	\node[vfvertex, Red] (C) at (6, 2) {};
	\node[vfvertex, ForestGreen] (D) at (4, 2) {};
	\node[vfvertex, Red] (E) at (4.7, 2.7) {};
	\node[vfvertex, Red] (F) at (1.5, 3) {};
	\node[vfvertex, Red] (G) at (9, 2) {};
	
	\draw[vfedge] (a)--(A);
	\draw[vfedge, Red] (f)--(A);
	\draw[vfedge] (g)--(A);
	\draw[vfedge] (a)--(B);
	\draw[vfedge] (b)--(B);
	\draw[vfedge] (g)--(B);
	\draw[vfedge] (h)--(B);
	\draw[vfedge] (e)--(F);
	\draw[vfedge, Red] (f)--(F);
	\draw[vfedge] (g)--(F);
	\draw[vfedge] (h)--(F);
	\draw[vfedge] (i)--(D);
	\draw[vfedge] (j)--(D);
	\draw[vfedge] (k)--(D);
	\draw[vfedge] (m)--(D);
	\draw[vfedge] (k) -- (E);
	\draw[vfedge, Red] (l) -- (E);
	\draw[vfedge] (m) -- (E);
	\draw[vfedge] (k) -- (C);
	\draw[vfedge, Red] (l) -- (C);
	\draw[vfedge] (j) -- (C);
	\draw[vfedge, Red] (c) -- (C);
	\draw[vfedge] (d) -- (C);
	\draw[vfedge] (b) to[out=10, in=270] (C);
	\draw[vfedge] (h) to[controls=+(300:2) and +(250:2)] (C);
	\draw[vfedge] (e) to[controls=+(-10:2) and +(90:2)] (C);
	\draw[vfedge] (i) to[controls=+(90:2) and +(100:2)] (C);
	\draw[vfedge] (a) to[controls=+(-70:2) and +(-90:4)] (G);
	\draw[vfedge] (b) to[controls=+(-50:2) and +(-100:3)] (G);
	\draw[vfedge, Red] (c) -- (G);
	\draw[vfedge] (d) -- (G);
	\draw[vfedge] (e) to[controls=+(50:2) and +(100:3)] (G);
	\draw[vfedge, Red] (f) to[controls=+(70:2) and +(90:4)] (G);
\end{tikzpicture}	
	\caption{The figure shows a graph $G$ where $V(G)$ are the filled vertices and $E(G)$ are the solid edges.
		Its vertex-face incidence graph $\vf(G)$ is the graph on all vertices (filled and empty) with the dotted edges.
		When choosing $T\subseteq V(G)$ to be the filled red vertices, $\vf(G)[\cF(G)\cup T]$ decomposes into three connected components, drawn in red, green, and blue, respectively.
		Hence $T$ is an $\cF$-transversal when choosing $\cF$ to be the faces of $G$ that are filled in red.
		Note that $\cF$ are exactly the odd faces of $G$; thus $T$ is an odd cycle transversal for $G$ by \cref{prop:vf_definition_of_transversal}.}
	\label{fig:vf}
\end{figure}

The notion of the vertex-face incidence graph allows us to define the following important concept:
\begin{definition}[$\cF$-Transversal]\label{def:transversal_of_face_set}
Let $G$ be a planar graph, embedded in the plane.
Let $\cF \subseteq \cF(G)$ have even cardinality $|\cF|$.
A set $T \subseteq V(G)$ is an \emph{$\cF$-transversal} if each connected component of $\vf(G)[\cF(G) \cup T]$ contains an even number of elements of $\cF$.
We define $\tau(\cF)$ to be the size of a minimum-cardinality $\cF$-transversal.
For an example see \Cref{fig:vf}.
\end{definition}

It is reasonable to call both vertex sets defined in \cref{def:odd_cycle_transversal} and \cref{def:transversal_of_face_set} \emph{transversals} since these notions coincide for $\cF := \fodd(G)$:
\begin{proposition}\label{prop:vf_definition_of_transversal}
Let $G$ be a planar graph with a fixed planar embedding.
A set $T \subseteq V(G)$ is an odd cycle transversal for $G$ if and only if $T$ is an $\fodd(G)$-transversal according to \cref{def:transversal_of_face_set}.
\end{proposition}
\begin{proof}
First we show ``$\Rightarrow$''.
Let $X \subseteq \cF(G) \cup T$ be a connected component of $\vf(G)[\cF(G) \cup T]$ containing an odd number of elements of $\fodd(G)$.
Then the symmetric difference $C := \bigoplus_{F \in X \cap \cF(G)} E(F)$ is odd and Eulerian and therefore contains some odd cycle.
However, if any edge $e \in E(G)$ is incident to some vertex in $T$, then the two faces adjacent to $e$ are connected in $\vf(G)[\cF(G) \cup T]$, so $e \notin E(C)$.
Thus, $V(C) \cap T = \emptyset$ and $T$ is no odd cycle transversal.

Next we show ``$\Leftarrow$''.
Assume there exists an odd cycle $C$ in $G$ with $V(C) \cap T = \emptyset$.
Let $\cF(C) \subseteq \cF(G)$ be the set of faces inside $C$ (w.r.t.\@ the planar embedding of $G$).
Clearly, no face in $\cF(C)$ can be connected to a face outside $\cF(C)$ in $\vf(G)[\cF(G) \cup T]$.
But $\cF(C)$ contains an odd number of elements of $\fodd(G)$ because $C = \bigoplus_{F \in \cF(C)} E(F)$ is odd.
Thus, one of the connected components of $\vf(G)[\cF(C) \cup T]$ must also contain an odd number of elements of $\fodd(G)$.
\end{proof}

\newpage

Our definition of $\cF$-transversals is closely related to the well-known notion of \emph{$T$-joins}:
\begin{definition}[$T$-join] Given a graph $G$ and a set $T \subseteq V(G)$, a $T$-join is a set $J \subseteq E(G)$ such that for any vertex $v \in V(G)$, $|\delta(v) \cap J|$ is odd if and only if $v \in T$.
\end{definition}

Clearly a $T$-join in $G$ exists if and only if each connected component of $G$ contains an even number of elements of $T$.
In particular, the definition of $\cF$-transversals from \cref{def:transversal_of_face_set} can be reformulated as $\vf(G)[\cF(G) \cup T]$ containing an $\cF$-join.

Using the fact that for $T, T' \subseteq V(G)$, a $T$-join $J$ and a $T'$-join $J'$, the symmetric difference $J\oplus J'$ is a $T\oplus T'$-join, we can show the following property of $\cF$-transversals:
\begin{proposition}\label{prop:symmetric_difference_of_transversals}
    Given a planar graph $G$ with a fixed planar embedding, let $\cF, \cF' \subseteq \cF(G)$ have even cardinality. Let $T\subseteq V(G)$ be an $\cF$-transversal and $T'\subseteq V(G)$ be an $\cF'$-transversal. Then $T\cup T'$ is an $\cF\oplus \cF'$-transversal.
\end{proposition}
\begin{proof}
    Let $J$ be an $\cF$-join contained in $\vf(G)[\cF(G)\cup T]$ and $J'$ be an $\cF'$-join contained in $\vf(G)[\cF(G)\cup T']$. Then $J\oplus J'$ is contained in $\vf(G)[\cF(G)\cup(T\cup T')]$. As $J\oplus J'$ is an $\cF\oplus \cF'$-join, this implies that $T\cup T'$ is an $\cF\oplus \cF'$-transversal.
\end{proof}

Similar to $\cF$-transversals, we define the packing number with respect to some face set:

\begin{definition}
 Given a planar graph $G$ with a set $\cF \subseteq \cF(G)$ of its faces, define $\nu(\cF)$ to be the maximum number of pairwise vertex-disjoint faces of $\cF$.
\end{definition}

\section{Proof of the Main Theorem}
\label{sec:thm}

\subsection{High-level outline}\label{sec:outline}

In this section we prove \cref{thm:main_thm}.
The outline roughly follows the proof of $\tau \leq 6 \nu$ by Kr\'al\textquoteright{}, Sereni and Stacho~\cite{KraSS12}.
They use the notion of \emph{clouds} of a planar graph $G$, which are the maximal subsets of $\fodd(G)$ that are connected in $\vf(G)[\fodd(G) \cup V(G)]$.
For any cloud, they compute a vertex set $T$ and a set $\cP$ of vertex-disjoint odd faces of the cloud such that $|T| \leq 6 |\cP| - 2$ and the whole cloud is contained in a single face of $G-T$.
They mark such a face as ``deadly''.
After applying this to each cloud, they use the following lemma to finish the proof:

\begin{lemma}[\cite{KraSS12}]\label{lemma:connect_clouds}
Let $G$ be a planar graph, embedded in the plane.
Let $\cF_\text{dead} \subseteq \cF(G)$ be a set of faces marked as \emph{deadly}, including all odd faces of $G$.
Then $\tau(G) \leq \nu_\text{dead}(G) + 2 |\cF_\text{dead}|$, where $\nu_\text{dead}(G)$ is the maximum number of pairwise vertex-disjoint odd cycles that are vertex-disjoint to all deadly faces in $G$.
\end{lemma}

For our approach it will be convenient to generalize the notion of clouds:

\begin{definition}\label{def:clouds}
 Let $G$ be a planar graph, embedded in the plane, and  $\cF^* \subseteq \cF(G)$ with $\fodd(G) \subseteq \cF^*$ define a set of \emph{special} faces of $G$.
 An \emph{$\cF^*$-cloud} of $G$ is a maximal set $\cW \subseteq \cF^*$ that is connected in $\vf(G)[\cF^* \cup V(G)]$, i.e., $\cW$ can be written as $\cW = X \cap \cF^*$, where $X$ is a connected component of $\vf(G)[\cF^* \cup V(G)]$.
\end{definition}

Instead of finding any packing of odd cycles in $G$, we will work with packings with a special structure:

\begin{definition}\label{def:special_packing}
Let $G$ be a planar graph, embedded in the plane.
Let $\cF^* \subseteq \cF(G)$ with $\fodd(G) \subseteq \cF^*$ be a set of special faces of $G$.
A set $\cP$ is called a \emph{special packing} for $G$ if each $C \in \cP$ is either a special face $C \in \cF^*$ or $C$ is an odd cycle in $G$ with $V(C) \cap F = \emptyset$ for all $F \in \cF^*$, such that any $C, C^\prime \in \cP$ are vertex-disjoint, i.e., $V(C) \cap V(C^\prime) = \emptyset$.
\end{definition}

Note that for $\cF^* = \fodd(G)$ any special packing for $G$ can be transformed into an odd cycle packing (i.e., a set of pairwise vertex-disjoint odd cycles) of the same size by replacing each $F \in \cF^* = \fodd(G)$ by some odd cycle in $E(F)$.

We will show that we can always find a special packing for $G$ and an odd cycle transversal of at most four times the size.
Our algorithm works as follows:
Observe that in the case where each $\cF^*$-cloud consists of a single face, Lemma~\ref{lemma:connect_clouds} directly yields a special packing and an odd cycle transversal of at most twice the size.
By a simple observation, we still get an odd cycle transversal of at most four times the size of a special packing in the case where any $\cF^*$-cloud $\cW$ has packing number $\nu(\cW) = 1$.
In the case of an $\cF^*$-cloud $\cW$ with packing number $\nu(\cW) > 1$ however, we will find a special face in $\cW$ with few neighbours, ``merge'' it with all its neighbours (and possibly one additional face) and recurse on the smaller instance.
We prove the existence of such a special face in $\cW$ in \cref{sec:cloud_lemma}.

\subsection{Detailed proof}
We first define a simple way to merge a set of faces into one face, without modifying the structure of the graph.

\begin{definition}\label{def:merging}
Let $G$ be a planar graph, embedded in the plane.
Let $F_1, F_2$ be faces of $G$ and $v \in V(F_1)\cap V(F_2)$.
A graph obtained after \emph{merging} the faces $F_1$ and $F_2$ along $v$ can be constructed as follows:
Let $e_1, \dots, e_k \in E(G)$ be the edges incident to $v$, ordered counterclockwise around $v$ (for an example, see \Cref{fig:merging}).
Let $j_1, j_2 \in \{1, \dots, k\}$ such that $F_i$ lies between $e_{j_i}$ and $e_{j_i+1}$ for $i=1,2$ (where we set $k + 1 := 1$).
W.l.o.g. $j_1 < j_2$.
Now split the vertex $v$ into two vertices $v^\prime$ and $v^{\prime\prime}$, connect the edges $e_{j_1 + 1}, \dots, e_{j_2}$ to $v^\prime$ and the remaining edges from $\delta_G(v)$ to $v^{\prime\prime}$.
\end{definition}

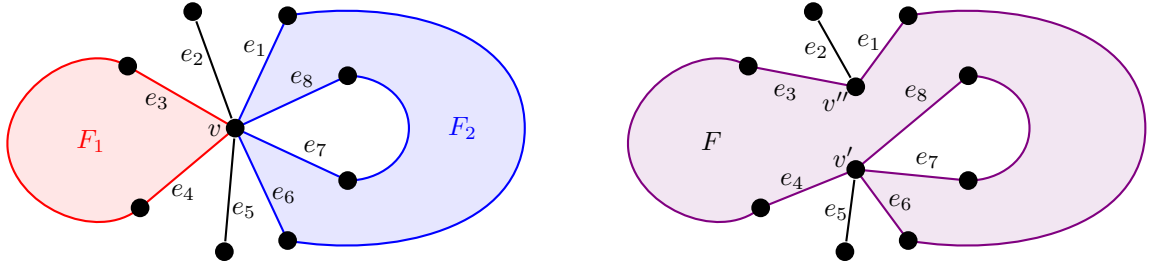
\begin{figure}
    \begin{center}
        \begin{tikzpicture}[scale=0.55, thick]
 \colorlet{color1}{red}
 \colorlet{color2}{blue}
 \colorlet{color3}{violet}
 \draw[color1, fill=color1!10] (0,0) to (150:3) to[out=150, in=95] (-5.5,-0.5) to[out=-85, in=220] (220:3) to (0,0);
 \draw[color2, fill=color2!10] (0,0) to (-25:3) to[out=-0, in=-90] (4.2,0) to[out=90, in=0] (25:3) to (0,0) to (65:3) to[out=10, in=90] (7,0) to[out=-90,in=-10] (-65:3) to (0,0);
 \node[fill, circle, inner sep=2.5, color=black] (v) at (0,0) {};
 \node[fill, circle, inner sep=2.5, color=black] (v1) at (65:3) {};
 \node[fill, circle, inner sep=2.5, color=black] (v2) at (110:3) {};
 \node[fill, circle, inner sep=2.5, color=black] (v3) at (150:3) {};
 \node[fill, circle, inner sep=2.5, color=black] (v4) at (220:3) {};
 \node[fill, circle, inner sep=2.5, color=black] (v5) at (-95:3) {};
 \node[fill, circle, inner sep=2.5, color=black] (v6) at (-65:3) {};
 \node[fill, circle, inner sep=2.5, color=black] (v7) at (-25:3) {};
 \node[fill, circle, inner sep=2.5, color=black] (v8) at (25:3) {};
 \draw (v) -- (v2);
 \draw (v) -- (v5);
 \node at (76:2) {$e_1$};
 \node at (121:2) {$e_2$};
 \node at (161:2) {$e_3$};
 \node at (231:2) {$e_4$};
 \node at (-84:2) {$e_5$};
 \node at (-54:2) {$e_6$};
 \node at (-14:2) {$e_7$};
 \node at (36:2) {$e_8$};
 \node at (185:0.5) {$v$};
 \node[color=color1] at (185:3.5) {$F_1$};
 \node[color=color2] at (0:5.5) {$F_2$};
 \begin{scope}[shift={(15,0)}]
  \draw[color3, fill=color3!10] (0,1) -- (150:3) to[out=150, in=95] (-5.5,-0.5) to[out=-85, in=220] (220:3) to (0,-1) to (25:3) to[out=0,in=90] (4.2,0) to[out=-90, in=0] (-25:3) -- (0,-1) -- (-65:3) to[out=-10,in=-90] (7,0) to[out=90,in=10] (65:3) -- (0,1);
  \node[fill, circle, inner sep=2.5, color=black] (vp) at (0,1) {};
  \node[fill, circle, inner sep=2.5, color=black] (vpp) at (0,-1) {};
 \node[fill, circle, inner sep=2.5, color=black] (v1) at (65:3) {};
 \node[fill, circle, inner sep=2.5, color=black] (v2) at (110:3) {};
 \node[fill, circle, inner sep=2.5, color=black] (v3) at (150:3) {};
 \node[fill, circle, inner sep=2.5, color=black] (v4) at (220:3) {};
 \node[fill, circle, inner sep=2.5, color=black] (v5) at (-95:3) {};
 \node[fill, circle, inner sep=2.5, color=black] (v6) at (-65:3) {};
 \node[fill, circle, inner sep=2.5, color=black] (v7) at (-25:3) {};
 \node[fill, circle, inner sep=2.5, color=black] (v8) at (25:3) {};
 \draw (vp) -- (v2);
 \draw (vpp) -- (v5);
 \node at (81:2.1) {$e_1$};
 \node at (117:2.1) {$e_2$};
 \node at (153:1.9) {$e_3$};
 \node at (220:2) {$e_4$};
 \node at (-103:2.1) {$e_5$};
 \node at (-59:2.1) {$e_6$};
 \node at (-24:1.9) {$e_7$};
 \node at (30:1.7) {$e_8$};
 \node at (-0.48,0.65) {$v^{\prime\prime}$};
 \node at (-0.3,-0.7) {$v^\prime$};
 \node at (185:3.5) {$F$};
 \end{scope}
\end{tikzpicture}
    \end{center}
    \caption{
        The left picture shows the faces $F_1$ and $F_2$, which meet at the vertex $v$, before merging.
        The edges incident to $v$ are ordered counterclockwise.
        Note that $j_1=3$, while for $j_2$ there are two options: $j_2=6$ and $j_2=8$. \\
        The right picture shows the graph obtained after merging $F_1$ and $F_2$ along $v$ for the choice $j_2 = 8$.
        The large face $F$ corresponds to the face set $\{F_1, F_2\}$.
    }
    \label{fig:merging}
\end{figure}

Let $G^\prime$ be the graph obtained by merging the faces $F_1$ and $F_2$ along $v$.
It is clear that $G^\prime$ contains exactly one face less than $G$:
All faces except for $F_1$ and $F_2$ remain unchanged, and $G^\prime$ contains one other face $F$ with $E(F) = E(F_1) \oplus E(F_2)$.
We say that $F$ corresponds to the face set $\{F_1, F_2\}$.
Also, all vertices except for $v$ remain unchanged, and we say that both $v^\prime$ and $v^{\prime\prime}$ in $V(G^\prime)$ correspond to $v \in V(G)$.
See \Cref{fig:merging}.

Given a set $\cF \subseteq \cF(G)$ of faces that are connected in $\vf(G)[\cF \cup V(G)]$, we can define a graph obtained after merging all faces in $\cF$ as follows:
Start by setting $\cF' := \cF$.
While $|\cF'| > 1$, choose two faces $F_1, F_2 \in \cF'$ and merge them via Definition~\ref{def:merging}.
Delete both $F_1$ and $F_2$ from $\cF'$ and add the new face in the constructed graph that corresponds to $\{F_1, F_2\}$ to $\cF'$.
In the graph resulting from all these merging operations, all faces except for the faces in $\cF$ remain unchanged and the faces from $\cF$ are replaced by exactly one additional face that corresponds to $\cF$.

\begin{lemma}
    \label{lem:induction_step}
    Let $G$ be a planar graph, embedded in the plane, and $\cF \subseteq \cF(G)$ a subset of its faces that is connected in $\vf(G)[\cF \cup V(G)]$.
    Let $G^\prime$ be a graph arising from $G$ when merging all faces of $\cF$ and let $T$ be an odd cycle transversal for $G^\prime$.
    Let $T_G \subseteq V(G)$ be constructed by replacing each $v \in T$ by the vertex in $G$ corresponding to $v$.
    Then there exists a subset $\cF'\subseteq \cF$ of even cardinality such that for any $\cF'$-transversal $T^\prime$, $T_G \cup T^\prime$ defines an odd cycle transversal for $G$.
\end{lemma}
\begin{proof}
Given a connected component $X$ of $\vf(G)[\cF(G) \cup T_G]$ with $X \cap \cF = \emptyset$, clearly also $X \cap V(\cF) = \emptyset$.
Thus, $X$ is also a connected component of $\vf(G^\prime)[\cF(G^\prime) \cup T]$ and therefore $|X \cap \fodd(G)|$ is already even.

Thus, we can construct a set $\cF^\prime \subseteq \cF$ by adding exactly one element from each connected component $X$ of $\vf(G)[\cF(G) \cup T_G]$ with $|X \cap \fodd(G)|$ odd.
In particular, each connected component of $\vf(G)[\cF(G) \cup T_G]$ contains an even number of elements of $\fodd(G) \oplus \cF^\prime$ and therefore $T_G$ is an $\fodd(G) \oplus \cF'$-transversal.

Now let $T^\prime \subseteq V(G)$ be an $\cF'$-transversal. According to \cref{prop:symmetric_difference_of_transversals}, $T_G \cup T^\prime$ is an $\fodd(G)$-transversal because $\fodd(G)\oplus \cF^\prime\oplus \cF^\prime = \fodd(G)$.
By \cref{prop:vf_definition_of_transversal}, $T_G \cup T^\prime$  hence is an odd cycle transversal for $G$.
\end{proof}

The following two lemmata show structural properties of $\cF^*$-clouds.
We defer their proofs to \cref{sec:cloud_lemma}.
\begin{restatable*}{lemma}{packingnumber}
	\label{lem:packing_number1}
	Let $G$ be a planar graph, embedded in the plane.
	Let $\cF \subseteq \cF(G)$ with $\nu(\cF) = 1$. Then for any $\cF^\prime \subseteq \cF$ of even cardinality we have $\tau(\cF^\prime) \leq 2$.
\end{restatable*}

\begin{restatable*}{lemma}{cloudlemma}
	\label{lem:clouds}
	Let $G$ be a planar graph, embedded in the plane, and $\cF^* \subseteq \cF(G)$ with $\fodd(G) \subseteq \cF^*$ a set of special faces.
	Let $\cW$ be an $\cF^*$-cloud of $G$ with $\nu(\cW) > 1$.
	Then there exists a set $\cF \subseteq \cW$ such that
	\begin{enumerate}[label=(\roman*)]
		\item \label{item:connected} $\cF$ is connected in $\vf(G)[\cF \cup V(G)|$ and
		\item \label{item:packingnumber} $\nu(\cF) > 1$ and
		\item \label{item:centralface} there is a face $F\in \cF$ such that all $F^\prime \in \cW$ where $V(F)\cap V(F') \neq \emptyset$ are also contained in $\cF$ and
		\item \label{item:transversalnumber} for any even-cardinality subset $\cF'\subseteq \cF$, we have $\tau(\cF') \leq 4$.
	\end{enumerate}
\end{restatable*}

We are now ready to show \cref{thm:main_thm}.
As explained in \cref{sec:outline}, we will show an even stronger statement by requiring a special packing:

\begin{theorem}\label{thm:main_thm_special}
Let $G$ be a planar graph, embedded in the plane.
Let $\cF^* \subseteq \cF(G)$ with $\fodd(G) \subseteq \cF^*$ be a set of special faces.
There exists a special packing $\cP^*$ for $G$ and an odd cycle transversal $T^*$ for $G$ such that $|T^*| \leq 4 |\cP^*|$.
\end{theorem}
\begin{proof}
We prove the theorem by induction on the number of faces $|\cF(G)|$.
Let $\clouds$ be the set of $\cF^*$-clouds of $G$.

As the base case of our induction we consider the case where $\nu(\cW) = 1$ for all clouds $\cW \in \clouds$.
Let $G^\prime$ arise from $G$ by merging the faces of $\cW$ for each $\cF^*$-cloud $\cW \in \clouds$.
Then, $G^\prime$ contains exactly one face per $\cF^*$-cloud $\cW \in \clouds$ and one face per non-special face of $G$.
We mark the faces corresponding to the $\cF^*$-clouds as ``deadly''.

By applying \cref{lemma:connect_clouds}, we can find an odd cycle transversal $T$ for $G^\prime$ and a set $\cP$ of pairwise vertex-disjoint odd cycles of $G^\prime$ such that $|T| \leq |\cP| + 2|\clouds|$ and no cycle in $\cP$ contains a vertex of a deadly face of $G^\prime$.
In particular, $\cP$ defines a special packing for $G^\prime$ and thus also for $G$.
By adding exactly one face of each $\cW \in \clouds$ to $\cP$ we get a special packing $\cP^*$ for $G$ with $|\cP^*| = |\cP| + |\clouds|$.

Let $T_G \subseteq V(G)$ arise by replacing each $v \in T$ by the vertex of $V(G)$ corresponding to $v$.
By applying \cref{lem:induction_step} and \cref{lem:packing_number1} to each $\cF^*$-cloud $\cW \in \clouds$, we can find a set $T_\cW$ with
$|T_\cW| \leq 2$ for each $\cW \in \clouds$ such that $T^* := T_G \cup \bigcup_{\cW \in \clouds} T_\cW$ is an odd cycle transversal for $G$ with
\[|T^*| \leq |T| + 2 |\clouds| \leq |\cP| + 4 |\clouds| \leq 4 |\cP^*|.\]
This concludes the case where $\nu(\cW) = 1$ for all $\cW \in \clouds$.

Next, assume that there exists an $\cF^*$-cloud $\cW \in \clouds$ with $\nu(\cW) > 1$.
Choose a subset $\cF \subseteq \cW$ as given by \cref{lem:clouds}.
Let $G^\prime$ arise from $G$ by merging the faces in $\cF$.
Note that $G'$ has fewer faces than $G$ because $\nu(\cF) > 1$ implies $|\cF| > 1$.
The set $\cF^{\prime *} \subseteq \cF(G^\prime)$ of special faces of $G^\prime$ can be constructed from $\cF^*$ by deleting the faces in $\cF$ and replacing them by the face $F_\cF$ of $G^\prime$ that corresponds to $\cF$.

By the induction hypothesis, we can find an odd cycle transversal $T$ and a special packing $\cP$ for $G^\prime$ such that $|T| \leq 4 |\cP|$.
Let $T_G \subseteq V(G)$ arise by replacing each $v \in T$ by the vertex of $V(G)$ corresponding to $v$.
By the properties of $\cF$ and Lemma~\ref{lem:induction_step} we can find a set $T_\cF \subseteq V(G)$ with $|T_\cF| \leq 4$ such that $T^* := T_G \cup T_\cF$ is an odd cycle transversal for $G$.
We will show how to find a special packing $\cP^*$ for $G$ with $|\cP^*| = |\cP| + 1$. This will finish the proof.

Clearly, $\cP \setminus \{F_\cF\}$ defines a special packing for $G$.
We consider two cases:
\begin{description}
    \item[Case 1:] $F_\cF \in \cP$.
    Since $\nu(\cF) \geq 2$, we can replace $F_\cF$ by two vertex-disjoint deadly faces of $\cF$ to get a special packing for $G$.
    \item[Case 2:] $F_\cF \notin \cP$.
    In this case $\cP$ already defines a special packing for $G$. Also, by \Cref{lem:clouds} \ref{item:centralface} there is a face $F \in \cF$ such that all special faces $F^\prime$ with $V(F) \cap V(F^\prime) \neq \emptyset$ are also contained in $\cF$.
    In particular, $\cP \cup \{F\}$ is a special packing for $G$. \qedhere
\end{description}
\end{proof}

As noted above, for $\cF^* := \fodd(G)$ any special packing for $G$ can be transformed into an odd cycle packing of the same size.
Furthermore, all steps in the proof of \cref{thm:main_thm_special} are constructive and can be carried out in polynomial time, including embedding $G$ in the plane (\cite{ChiNAO85}).
Thus we get \cref{thm:main_thm} as a corollary:

\begin{corollary}
	Let $G$ be an undirected planar graph.
	Then there exists an odd cycle transversal $T$ and a set $\cP$ of pairwise vertex-disjoint odd cycles in $G$ with $|T|\leq 4|\cP|$.
	$T$ and $\cP$ can be computed in polynomial time.
\end{corollary}

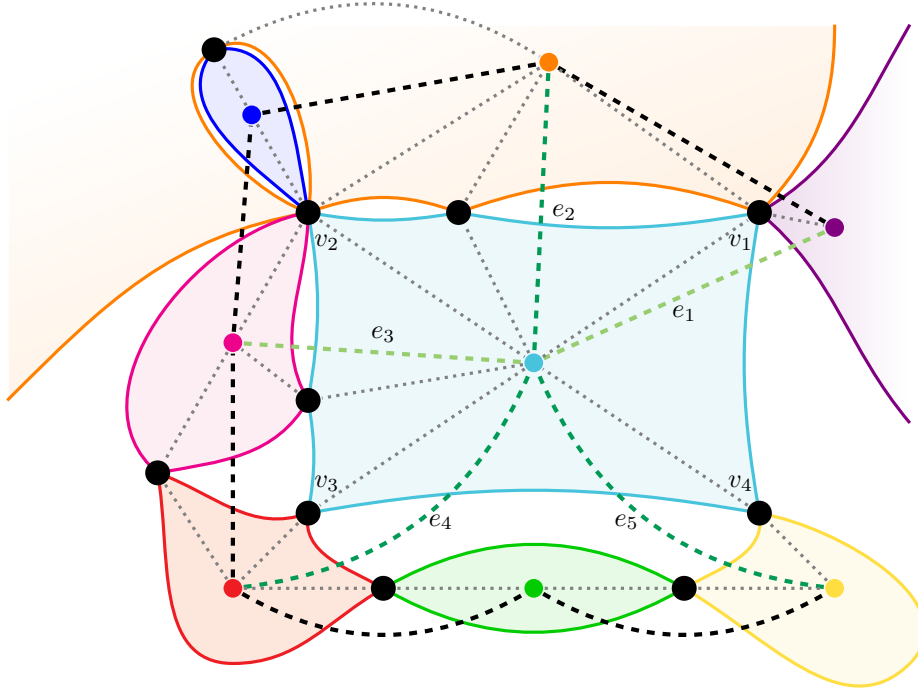
\begin{figure}[h]
    \begin{center}
          \begin{tikzpicture}[very thick, scale=1]
  	\fill[draw=SkyBlue, fill = SkyBlue!10] (0,0)to [out=-80, in=80] (0,-2.5) to [out=-80, in=80] (0, -4) to [out=10, in =170] (6, -4) to [out=100, in=-100] (6,0) to [out=190, in = -10] (2,0) to [out=190, in = -10] (0,0);
  	
	\fill[draw=white, left color=violet!15, right color=white] (8,2.5) to[out=-120,in=30] (6,0) to[out=-40,in=130] (8,-2.8) -- (8,2.5);
	\draw[violet] (8,2.5) to[out=-120,in=30] (6,0) to[out=-40,in=130] (8,-2.8);
	
	\fill[draw=white, top color = white, bottom color=orange!15, shading angle=20] (-4,-2.5) to[out=45, in=-170] (0, 0) to[out=20, in=160] (2,0) to[out=20, in=160] (6,0) to[out=45, in=-90] (7,2.5)  -- (-4, 2.5) -- (-4, -2.5);
	\fill[white] (0,0) to [out=80, in=30] (120:2.5) to [out=-150, in=160] (0,0);
	\draw[orange] (-4,-2.5) to[out=45, in=-170] (0, 0) to[out=80, in=30] (120:2.5) to[out=-150, in=160] (0, 0) to[out=20, in=160] (2,0)  to[out=20, in=160] (6,0) to[out=45, in=-90] (7,2.5);
	
	\draw[blue, fill=blue!8] (0,0) to[out=100, in=10] (120:2.5) to[out=-130, in=140] (0, 0);
	
	\draw[Magenta, fill=Magenta!8] (0,0) to[out=-170, in=140] (-120:4) to [out=10, in = -120] (0, -2.5) to[out=120, in=-90] (0,0);
	
	\draw[Red, fill=Red!10] (-120:4) to[out=-20, in=200] (0, -4)  to[out=-100, in=150] (1, -5) to [out=210, in=0] (-1, -6) to [out=180, in=-80]  (-120:4);
	
	\draw[Goldenrod, fill=Goldenrod!10] (6, -4) to[out=-10, in=45] (8, -6) to[out=225, in=-30] (5, -5) to[out=30, in=-80](6,-4);
	
	\draw[darkgreen, fill=darkgreen!10] (1, -5) to[out=30, in=150] (5,-5) to [out=210, in=-30] (1, -5);
	
	\node[fill, circle, minimum size=3mm] (a) at (0,0) {};
	\node[fill, circle, minimum size=3mm] (b) at (0,-4) {};
	\node[fill, circle, minimum size=3mm] (c) at (6,-4) {};
	\node[fill, circle, minimum size=3mm] (d) at (6,0) {};
	\node[fill, circle, minimum size=3mm] (e) at (120:2.5) {};
	\node[fill, circle, minimum size=3mm] (f) at (-120:4) {};
	\node[fill, circle, minimum size=3mm] (g) at (1,-5) {};
	\node[fill, circle, minimum size=3mm] (h) at (5,-5) {};
	\node[fill, circle, minimum size=3mm] (i) at (2,0) {};
	\node[fill, circle, minimum size=3mm] (j) at (0, -2.5) {};
	
	\node at (5.75, -0.4) {$v_1$};
	\node at (0.25, -0.4) {$v_2$};
	\node at (0.25, -3.6) {$v_3$};
	\node at (5.75, -3.6) {$v_4$};

	\node[fill, circle, inner sep=2.5, color=Magenta] (c1) at (-120:2) {};
	\node[fill, circle, inner sep=2.5, color=blue] (c3) at (120:1.5) {};
	\node[fill, circle, inner sep=2.5, color=orange] (c4) at (3.2, 2) {};
	\node[fill, circle, inner sep=2.5, color=violet] (c5) at (7,-0.2) {};
	\node[fill, circle, inner sep=2.5, color=SkyBlue] (c6) at (3, -2) {};
	\node[fill, circle, inner sep=2.5, color=Red] (c7) at (-1, -5) {};
	\node[fill, circle, inner sep=2.5, color=darkgreen] (c8) at (3, -5) {};
	\node[fill, circle, inner sep=2.5, color=Goldenrod] (c9) at (7, -5) {};
	
	\colorlet{halfcharge}{YellowGreen}
	\colorlet{fullcharge}{ForestGreen}
	\draw[dashed, ultra thick] (c1) to (c3);
	\draw[dashed, ultra thick] (c3) to (c4);
	\draw[dashed, ultra thick] (c4) to (c5);
	\draw[dashed, ultra thick, halfcharge] (c5) to node [below, text=black] {$e_1$} (c6);
	\draw[dashed, ultra thick, halfcharge] (c1) to node [above, text=black] {$e_3$}(c6);
	\draw[dashed, ultra thick, fullcharge] (c4) to node [right, text=black] {$e_2$} (c6);
	\draw[dashed, ultra thick] (c7) to (c1);
	\draw[dashed, ultra thick, bend right] (c7) to (c8);
	\draw[dashed, ultra thick, bend right] (c8) to (c9);
	\draw[dashed, ultra thick, bend right, fullcharge] (c7) to node [right, text=black] {$e_4$} (c6);
	\draw[dashed, ultra thick, bend left, fullcharge] (c9) to node [left, text=black] {$e_5$} (c6);
	
	\draw[dotted, gray] (a) -- (c6) -- (b) -- (c7) -- (g) -- (c8) -- (h) -- (c9) -- (c) -- (c6) -- (d) -- (c5);
	\draw[dotted, gray] (d) -- (c4) -- (a) -- (c3) -- (e);
	\draw[dotted, gray, bend left] (e) to (c4);
	\draw[dotted, gray] (a) -- (c1) -- (f) -- (c7);
	\draw[dotted, gray] (c6) -- (i) -- (c4);
	\draw[dotted, gray] (c6) -- (j) -- (c1);
\end{tikzpicture}
    \end{center}
    \caption{
        Example for a situation as in \cref{def:reduced_conflict_graph}, showing the surroundings of the light blue face of a graph $G$:
        The colored faces are the faces in $\cF\subseteq \cF(G)$.
        The relevant vertices of $G$ are drawn in black.
        The gray dotted lines are the edges in the vertex-face incidence graph $\vf(G)$ that are incident to $\cF$ and the black vertices.
        Then the reduced conflict graph $R$ for $\cF$ is constructed on the colored vertices (each representing a face in $\cF$).
        The edges of $R$ are the thick dashed lines (black and light and dark green).
        Note that $R$ depends on the embedding of $\vf(G)$ since there are two possibilities for embedding the edge between $v_2$ and the orange face.\\
        When we choose the light blue face as $F\in \cF$ in \cref{lem:degree_in_reduced_conflict_graph}, we get $T_1 = \{v_1, v_2\}$.
        Hence $\cF_1$ consists of the light blue, the violet, the orange, the dark blue, and the magenta face.
        $\cF_2$ then consists of the red and the yellow face; we get $T_2 = \{v_3, v_4\}$. 
        The edges in $\delta_R(F)$ that ultimately carry a charge of $\frac12$ are colored in light green; the edges with a charge of $1$ are colored in dark green:
        $v_1$ adds a charge of $\frac12$ to each of $e_1$ and $e_2$, $v_2$ adds a charge of $\frac12$ to each of $e_2$ and $e_3$, and $v_3$ and $v_4$ add a charge of 1 to $e_4$ and $e_5$, respectively.
    }
    \label{fig:reduced_conflict_graph}
\end{figure}

\section{Structural properties of clouds}\label{sec:cloud_lemma}
The goal of this section is to prove the statements about clouds that we used in our proof of \cref{thm:main_thm_special}.
We first define a ``reduced'' planar version of the conflict graph on the faces of an $\cF^*$-cloud.
A similar construction was used in \cite{Sch24}:
\begin{definition}[Reduced conflict graph]\label{def:reduced_conflict_graph}
Given a planar graph $G$ with a fixed planar embedding and a set $\cF \subseteq \cF(G)$ of faces of $G$, we define the \emph{reduced conflict graph} $R$ for $\cF$ as a planar graph on $V(R) := \cF$ as follows:
For a vertex $v\in V(G)$ let $\cF_v\subseteq \cF$ be the set of faces containing $v$ and $k:= |\cF_v|$.
The planar embedding of $\vf(G)$ induces a cyclic order on $\delta_{\vf(G)}(v) = \{\{v, F\} : F \in \cF_v\}$.
Enumerate the faces in $\cF_v$ according to this order as $F_1 =: F_{k + 1}, F_2,\dots, F_k$.
If $k \geq 2$, add the edges $\{F_i, F_{i+1}\}$ for any $i = 1,\dots, k$ to $R$ with the obvious planar embedding.
Finally, we identify homotopic edges in $R$ (i.e., parallel edges that bound a face of $R$).
For an example see \Cref{fig:reduced_conflict_graph}.
\end{definition}

The following lemma helps us formalize some intuition about the reduced conflict graph, see also \Cref{fig:reduced_conflict_graph}.
\begin{lemma}
    \label{lem:degree_in_reduced_conflict_graph}
    Let $G$ be a planar graph with a fixed planar embedding, $\cF\subseteq \cF(G)$ a set of faces of $G$, and $R$ the reduced conflict graph for $\cF$.
    Given a face $F \in \cF$, there is a subset of vertices $T\subseteq V(F)$  with $|T|\leq |\delta_R(F)|$ such that $T\cap V(F') \neq \emptyset$ for all $F'\in \cF$ with $V(F')\cap V(F)\neq \emptyset$.
\end{lemma}
\begin{proof}
	We will define non-negative charges $c(e)$ on the edges $e\in\delta_R(F)$ as follows.
    Let $T_1 \subseteq V(F)$ be the set of vertices of $F$ that are contained in at least three faces in $\cF$.
    When constructing the reduced conflict graph $R$ for $\cF$, each $v\in T_1$ produces a face $B_v$ of $R$ that is incident to $F$ in $R$.
    For every such face $B_v$, add a charge of $\frac12$ to two different edges in $\delta_R(F)$ that are boundary edges of $B_v$.
    Because each edge is only incident to at most two faces, $c(e) \leq 1$ for all $e\in \delta_R(F)$. See also \Cref{fig:reduced_conflict_graph}.
    
    Define $\cF_1 := \{F'\in \cF: V(F')\cap T_1 \neq \emptyset\}$ and $\cF_2 := \{F'\in \cF: V(F')\cap V(F) \neq \emptyset\} \setminus \cF_1$.
    Construct $T_2\subseteq V(F)$ by adding exactly one vertex $v'\in V(F')\cap V(F)$ for each $F'\in \cF_2$.
    Now consider a face $F'\in \cF_2$, and the corresponding $v'\in V(F') \cap V(F) \cap T_2$.
    Since $v'\notin T_1$, $F$ and $F'$ are the only faces in $\cF$ containing $v'$.
    Hence $R$ contains an edge $e_{v'} = \{F, F'\} \in \delta_R(F)$.
    Note that $e_{v'}$ is not incident to any face $B_v$ for any $v\in T_1$ because otherwise we would have $v\in V(F')$ for this $v\in T_1$.
    Therefore $e_{v'}$ has not been charged at all so far.
    Now $v'$ adds a charge of 1 to $e_{v'}$.
    In doing so, we still guarantee $c(e) \leq 1$ for all $e\in \delta_R(F)$.
    
    Ultimately we obtain $|\delta_R(F)|\geq \sum_{e\in\delta_R(F)}c(e) = |T_1| + |T_2|$ because every vertex in $T_1$ and every vertex in $T_2$ has added a total charge of 1 to the edges in $\delta_R(F)$.
    But $T := T_1\cup T_2$ is a vertex set as required in the lemma.
\end{proof}

We can now prove \cref{lem:packing_number1} which we restate here for convenience:

\packingnumber

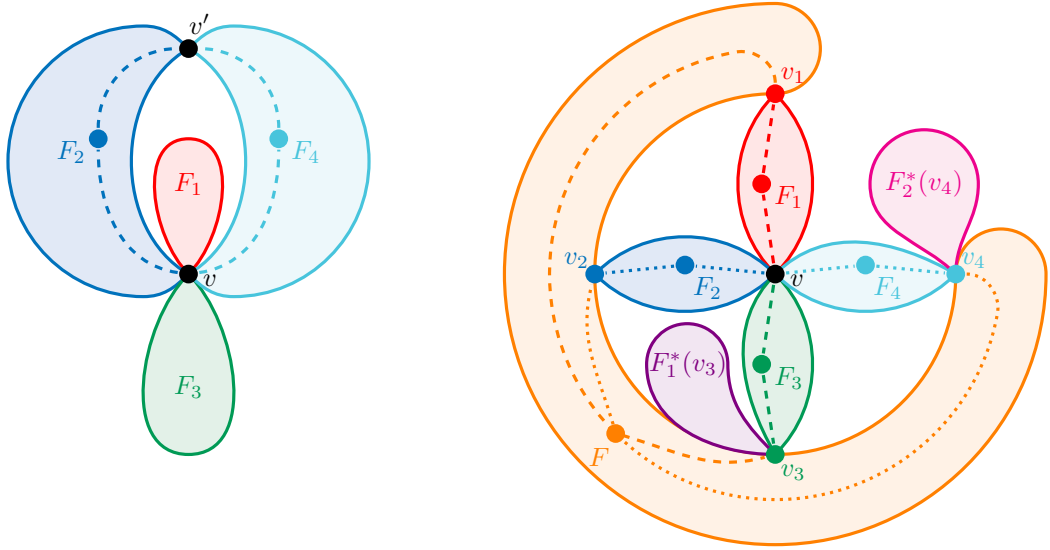
\begin{figure}
	\begin{center}
		\begin{tikzpicture}[scale=0.6, very thick]
     \colorlet{colorf1}{red}
     \colorlet{colorf2}{RoyalBlue}
     \colorlet{colorf3}{ForestGreen}
     \colorlet{colorf4}{SkyBlue}
     \colorlet{colorf}{orange}
     \colorlet{colorfstar1}{violet}
     \colorlet{colorfstar2}{magenta}
    \begin{scope}[shift = {(13,0)}]
     \draw[colorf1, fill=colorf1!10] (0,0) to[out=45, in=-45] (0,4) to[out=225, in=135] (0,0);
     \draw[colorf2, fill=colorf2!10] (0,0) to[out=135, in=45] (-4,0) to[out=-45, in=-135] (0,0);
     \draw[colorf3, fill=colorf3!10] (0,0) to[out=-135, in=120] (0, -4) to[out=45, in=-45] (0,0);
     \draw[colorf4, fill=colorf4!10] (0,0) to[out=-45, in=-135] (4,0) to[out=150, in=45] (0,0);
     \draw[colorf, fill=colorf!10] (4,0) arc (0:-270:4) arc (-90:90:1) arc (90:360:6) arc (0:180:1);
     \draw[colorfstar1, fill=colorfstar1!10] (0,-4) to[out=135, in=-90] (-1.05,-2) arc (0:180:0.9) to[out=-90, in=180] (0,-4);
     \draw[colorfstar2, fill=colorfstar2!10] (4,0) to[out=90, in=-90] (4.5,2) arc (0:180:1.2) to[out=-90, in=135] (4,0);
     
     \node[fill, circle, inner sep=2.5, color=black] (v) at (0,0) {};
     \node[fill, circle, inner sep=2.5, color=colorf1] (f1) at (-0.3,2) {};
     \node[fill, circle, inner sep=2.5, color=colorf2] (f2) at (-2,0.2) {};
     \node[fill, circle, inner sep=2.5, color=colorf3] (f3) at (-0.3,-2) {};
     \node[fill, circle, inner sep=2.5, color=colorf4] (f4) at (2,0.2) {};
     \node[fill, circle, inner sep=2.5, color=colorf1] (v1) at (0,4) {};
     \node[color=colorf1] at (0.4,4.4) {$v_1$};
     \node[fill, circle, inner sep=2.5, color=colorf2] (v2) at (-4,0) {};
     \node[color=colorf2] at (-4.4,0.3) {$v_2$};
     \node[fill, circle, inner sep=2.5, color=colorf3] (v3) at (0,-4) {};
     \node[color=colorf3] at (0.4,-4.4) {$v_3$};
     \node[fill, circle, inner sep=2.5, color=colorf4] (v4) at (4,0) {};
     \node[color=colorf4] at (4.4,0.3) {$v_4$};
     \node[fill, circle, inner sep=2.5, color=colorf] (f) at (225:5) {};
     \node[color=colorf1] at (0.3,1.7) {$F_1$};
     \node[color=colorf2] at (-1.5,-0.3) {$F_2$};
     \node[color=colorf3] at (0.3,-2.3) {$F_3$};
     \node[color=colorf4] at (2.5,-0.3) {$F_4$};
     \node[color=colorfstar1] at (-1.93, -2) {$F^*_1(v_3)$};
     \node[color=colorfstar2] at (3.3,2) {$F^*_2(v_4)$};
     \node[color=colorf] at (-3.9,-4.05) {$F$};
     \node[color=black] at (0.48,-0.14) {$v$};
     
     \draw[dashed, colorf1] (v) to (f1) to (v1);
     \draw[dashed, colorf3] (v) to (f3) to (v3);
     \draw[dotted, colorf2] (v) to (f2) to (v2);
     \draw[dotted, colorf4] (v) to (f4) to (v4);
     \draw[dashed, colorf] (v1) to[out=90,in=10] (100:5) arc (100:225:5) to[out=-15,in=195] (v3);
     \draw[dotted, colorf] (v4) to[out=0,in=80] (-10:5) arc (350:225:5) to[out=105, in=-105] (v2);
    \end{scope}

	\begin{scope}[shift={(0,0)}]
	 \draw[colorf1, fill=colorf1!10] (0,0) to[out=45, in=0] (0,3) to[out=180, in=135] (0,0);
     \draw[colorf2, fill=colorf2!10] (0,0) to[out=150, in=-150] (0,5) to[out=135, in=0] (-1,5.5) arc (90:270:3) to[out=0, in=-135] (0,0);
     \draw[colorf3, fill=colorf3!10] (0,0) to[out=-135, in=180] (0, -4) to[out=0, in=-45] (0,0);
     \draw[colorf4, fill=colorf4!10] (0,0) to[out=-45, in=180] (1,-0.5) arc(-90:90:3) to[out=180,in=45] (0,5) to[out=-30, in=30] (0,0);
     \node[fill, circle, inner sep=2.5, color=black] (v) at (0,0) {};
     \node[fill, circle, inner sep=2.5, color=colorf2] (f2) at (-2,3) {};
     \node[fill, circle, inner sep=2.5, color=colorf4] (f4) at (2,3) {};
     \node[fill, circle, inner sep=2.5, color=black] (vprime) at (0,5) {};
     \node[color=colorf1] at (0,2) {$F_1$};
     \node[color=colorf2] at (-2.6,2.7) {$F_2$};
     \node[color=colorf3] at (0,-2.5) {$F_3$};
     \node[color=colorf4] at (2.6,2.7) {$F_4$};
     \node[color=black] at (0.48,-0.14) {$v$};
     \node[color=black] at (0.25,5.55) {$v^\prime$};
     \draw[dashed, color=colorf2] (v) to[out=172, in=-90] (f2) to[out=90,in=180] (vprime);
     \draw[dashed, color=colorf4] (v) to[out=8,in=-90] (f4) to[out=90,in=0] (vprime);
	\end{scope}
\end{tikzpicture}
	\end{center}
	\caption{
		The left part shows the situation of Case 1 in Lemma~\ref{lem:packing_number1} where $F_2$ and $F_4$ meet in $v^\prime$.
		The cycle $C$ in $\vf(G)$ is drawn dashed.
		The right part shows the more complicated Case 2.
		The embeddings of the cycles $C_1$ (dashed lines) and $C_2$ (dotted lines) both separate the faces $F^*_1(v_3)$ and $F^*_2(v_4)$ from each other.
	}
	\label{fig:nu1_lemma}
\end{figure}

\begin{proof}
	First, observe that $\nu(\cF) = 1$ implies that for any $F, F^\prime \in \cF$ we have $V(F) \cap V(F^\prime) \neq \emptyset$.
	Thus, the statement clearly holds if $|\cF| \leq 4$:
	Given an even-cardinality subset $\cF^\prime \subseteq \cF$, partition $\cF^\prime$ into at most two pairs.
	For each pair, choosing a vertex where the two paired up faces meet yields an $\cF^\prime$-transversal.
	
	So we assume $|\cF| \geq 5$ from now on.
	In this case, we will show something stronger:
	Call a set $T \subseteq V(\cF)$ \emph{good} if $\vf(G)[T \cup \cF]$ is connected.
	In particular, a good set is an $\cF^\prime$-transversal for any even-cardinality $\cF^\prime \subseteq \cF$.
	We will show the existence of a good set $T$ with $|T| \leq 2$.
	
	Let $G^\prime := \vf(G)[\cF \cup V(G)]$ with a planar embedding as in \cref{def:vf}.
	If there is no vertex $v \in V(G)$ where more than $3$ faces of $\cF$ meet, then we can embed the conflict graph of $\cF$, i.e.\@ the complete graph on $\cF$, planarly, which is a contradiction.
	Thus, choose $v \in V(G)$ with four different faces $F_1, \dots, F_4 \in \cF$ such that the edges $\{v, F_1\}, \{v, F_2\}, \{v, F_3\}, \{v, F_4\} \in \delta_{G^\prime}(v)$ are arranged in this order around $v$ (in the planar embedding of $G^\prime$).
	
	\begin{description}
		\item[Case 1:] For some $i \in \{1, 2\}$, $F_i$ and $F_{i+2}$ meet in another vertex $v^\prime \neq v$.
		W.l.o.g.\@ $i=2$.
		We consider the cycle $C = v F_2 v^\prime F_4 v$ in $G^\prime$.
		Both connected components of $\mathbb{R}^2 \setminus C$ contain a face in $\cF$, namely $F_1$ and $F_3$, respectively (cf.\@ \Cref{fig:nu1_lemma}).
		Thus, any other cycle in $\cF$ must contain $v$ or $v^\prime$ in order to intersect the vertex sets of both $F_1$ and $F_3$.
		In particular, $T := \{v, v'\}$ is good.
		\item[Case 2:] $V(F_1) \cap V(F_3) = \{v\}$ and $V(F_2) \cap V(F_4) = \{v\}$.
		We can assume the existence of a face $F \in \cF$ with $v \notin V(F)$ because otherwise $T := \{v\}$ is good.
		For $i = 1, \dots, 4$ let $v_i \in V(G)$ such that $F$ meets $F_i$ in $v_i$.
		As in Case 1, the cycle $C_1 = v F_1 v_1 F v_3 F_3 v$ in $G^\prime$ has the property that both connected components of $\mathbb{R}^2 \setminus C_1$ contain a face in $\cF$, namely $F_2$ and $F_4$, respectively.
		Analogously, $F_1$ and $F_3$ are part of different connected components of $\mathbb R^2\setminus C_2$ for the cycle $C_2 = v F_2 v_2 F v_4 F_4 v$ in $G'$.
		See \Cref{fig:nu1_lemma}.
		Let $T_1 := V(C_1)\cap V(G) = \{v, v_1, v_3\}$ and $T_2 := V(C_2)\cap V(G) = \{v, v_2, v_4\}$.
		
		We will show that some two-element subset $T$ of $T_1$ or $T_2$ is good.
		Assume this is false, i.e., for any $i \in \{1, 2\}$ and any $t \in T_i$ there is $F^*_i(t) \in \cF$ such that $V(F^*_i(t)) \cap T_i = \{t\}$.
		W.l.o.g.\@ $F^*_1(v) = F_2$ and $F^*_2(v) = F_1$.
		
		Now we consider $F^*_1(v_3)$.
		We know that $F_2 = F^*_1(v)$ and $F^*_1(v_3)$ share a vertex, but $v, v_1\notin V(F^*_1(v_3))$ and $v_1, v_3\notin V(F_2)$.
		Hence $F_2$ and $F^*_1(v_3)$ must lie in the same connected component of $\mathbb{R}^2 \setminus C_1$.
		Also, $F^*_1(v_3)$ lies in the same connected component of $\mathbb{R}^2 \setminus C_2$ as $F_3$ because $v_3 \in V(F^*_1(v_3)) \cap V(F_3)$, but $v_3\notin C_2$.
		
		Analogously, $F_2^*(v_4)$ and $F_4$ lie in the same connected component of $\mathbb R^2 \setminus C_1$, and $F_2^*(v_4)$ and $F_1$ lie in the same connected component of $\mathbb R^2 \setminus C_2$.
		
		Since $C_2$ separates $F_1$ from $F_3$, and $C_1$ separates $F_2$ from $F_4$, this implies that $F^*_2(v_4)$ lies in the other connected component of both $\mathbb{R}^2 \setminus C_1$ and $\mathbb{R}^2 \setminus C_2$ than $F^*_1(v_3)$.
		In particular, a vertex where $F^*_1(v_3)$ and $F^*_2(v_4)$ meet must lie in $V(C_1) \cap V(C_2) \cap V(G) = T_1 \cap T_2 = \{v\}$, a contradiction. \qedhere
	\end{description}
\end{proof}

We finally prove \cref{lem:clouds}.
To this end, we want to find a face $F \in \cW$ such that for any even-cardinality subset $\cF'$ of the set of $F$ and all its adjacent faces in $\cW$ we can find a small $\cF'$-transversal.
We will see that it suffices to find $F\in\cW$ with degree at most 4 in the reduced conflict graph.
If however there is no such $F$, we can instead find $F\in \cW$ with degree 5 in the reduced conflict graph and at least four incident faces that are triangles.
It turns out that such an $F$ also yields the desired outcome.
To simplify notation, we define the \emph{degree of a face}:
\begin{definition}[Boundary edge multi-set, degree of a face]
    For a face $F$ of a planar graph $G$ embedded in the plane, we denote by $E^+(F)$ the multi-set of boundary edges of $F$, containing bridges twice.
    We define the \emph{degree} of $F$ to be $\deg(F) := |E^+(F)|$.
\end{definition}

The following lemma follows from Euler's formula.
A slightly stronger statement has already been proven in~\cite{CheFS12} and~\cite{MaYZ13} (while working with the planar dual of $G$).
Here, we only briefly show what is necessary for proving Lemma~\ref{lem:clouds}.

\begin{lemma}
    \label{lem:min_degree}
    Let $G$ be a planar graph without any homotopic edges and without self-loops.
    Then there exists a vertex $v$ such that either
    \begin{enumerate}
        \item $|\delta_G(v)| \leq 4$ or
        \item $|\delta_G(v)| = 5$ and $v$ has at least four incident faces of degree 3.
    \end{enumerate}
\end{lemma}
\begin{proof}
Assume the statement is false, i.e.\@, assume that $G$ has no vertex $v$ with $|\delta_G(v)| \leq 4$ and that each vertex $v$ with $|\delta_G(v)| = 5$ has at most three incident faces with degree 3.

For a vertex $v$ and a face $F$, define $\iota(v,F) := \frac12\cdot |\{e \in E^+(F) : v \in e\}|$.
Note that for any vertex $v$, we have $\sum_{F\in \cF(G)} \iota(v, F) = |\delta_G(v)|$.
Similarly, for a given face $F$, we have $\sum_{v\in V(G)} \iota(v, F) = \deg(F)$.

We assign a non-negative \emph{charge} $c(F)$ to each face $F$ as follows:
Each vertex $v$ with $|\delta_G(v)| = 5$ adds a charge of $\frac 14 \cdot \iota(v,F)$ to each face $F$ with $\deg(F)\geq 4$.
Since a vertex $v$ with $|\delta_G(v)| = 5$ has at most three incident faces with degree 3, $v$ adds at least a charge of $2\cdot \frac 14$ in total.
Let $a$ be the number of vertices of $G$ with degree exactly 5.
Then $\sum_{F\in \cF(G)}c(F) \geq  \frac 12\cdot a$.

Moreover on the one hand if $\deg(F) = 3$ then $c(F) = 0$; on the other hand  $\deg(F) \geq 4$ implies $\deg(F) -3 \geq \frac14\deg(F) \geq c(F)$. Note that each face has degree at least 3 because $G$ does not have any homotopic edges. Hence we have for each face $F$: $\deg(F) \geq 3 + c(F)$.

Let $m = |E(G)|$, $n = |V(G)|$, and $f = |\cF(G)|$. Euler's formula gives us
\begin{equation}
    \label{eqn:eulers_formula}
    f = m +2 -n.
\end{equation}
Furthermore we have
\begin{equation}
    \label{eqn:vertex_degree_sum}
    2m = \sum_{v\in V(G)}|\delta_G(v)| \geq 5a + 6(n-a) = 6n - a
\end{equation}
and
\begin{equation}
    \label{eqn:face_degree_sum}
    2m = \sum_{F\in \cF(G)} \deg(F)\geq \sum_{F\in \cF(G)} (3 + c(F)) \geq 3f + \frac12a.
\end{equation}
Plugging in (\ref{eqn:eulers_formula}) into (\ref{eqn:face_degree_sum}) yields $2m\geq 3m + 6 -3n + \frac 12 a$ and thus $m\leq 3n-6-\frac12a$.
Together with (\ref{eqn:vertex_degree_sum}) we hence obtain
\begin{equation*}
    6n -a \leq 2m \leq 6n - 12 - a
\end{equation*}
which yields a contradiction. Thus our initial assumption must be wrong; this proves the lemma.
\end{proof}

\cloudlemma

\begin{figure}
	\centering
	\begin{tikzpicture}[thick]
 \begin{scope}
 
 \draw[cyan, fill = cyan!10, very thick] (1,0) to [out = 160, in = -70] (0,1) to [out = -110, in = 20] (-1, 0) to [out = -20, in = 110] (0, -1) to [out = 70, in = -160] (1,0);
 
 \draw[orange, fill = orange!10, very thick] (1,0) to [out = 80, in = -10] (0.5, 1.5) to [out = 10, in = 135] (2,2) to [out=-45, in = 60] (1,0);
 \draw[red, fill = red!10, very thick] (1,0) to [out = 40, in = 120] (3.5, 1) to [out = -60, in = 10] (1,0);
 \draw[red, fill = red!10, very thick] (1,0) to [out = -40, in = -120] (3.5, -1) to [out = 60, in = -10] (1,0);
 \draw[orange, fill = orange!10, very thick] (1,0) to [out = -80, in = 10] (0.5, -1.5) to [out = -10, in = -135] (2,-2) to [out=45, in = -60] (1,0);
 
 \draw[ForestGreen, fill = ForestGreen!10, very thick] (0,1) to [out = 70, in = -160] (0.5, 1.5) to [out = 160, in = 45] (-1, 2) to [out = 225, in = 90] (-1, 1) to [out = 10, in = 170] (0,1);
 
 \draw[ForestGreen, fill = ForestGreen!10, very thick] (-1, 0) to [out = 100, in = -100] (-1, 1) to [out = 180, in = 90] (-2, 0) to [out = -90, in = 180] (-1, -1) to [out = 100, in = -100] (-1, 0);
 
 \draw[ForestGreen, fill = ForestGreen!10, very thick] (0,-1) to [out = -70, in = 160] (0.5, -1.5) to [out = -160, in = -45] (-1, -2) to [out = -225, in = -90] (-1, -1) to [out = -10, in = -170] (0,-1);
 
 \node[circle, minimum size=2mm, fill=cyan] (F) at (0,0) {};
 \node[circle, minimum size=2mm, fill=red] (f1) at (2.5,-0.7) {};
 \node[circle, minimum size=2mm, fill=red] (f2) at (2.5,0.7) {};
 \node[circle, minimum size=2mm, fill=orange] (f3) at (1.5,1.5) {};
 \node[circle, minimum size=2mm, fill=ForestGreen] (f4) at (-0.5,1.5) {};
 \node[circle, minimum size=2mm, fill=ForestGreen] (f5) at (-1.5,0) {};
 \node[circle, minimum size=2mm, fill=ForestGreen] (f6) at (-0.5,-1.5) {};
 \node[circle, minimum size=2mm, fill=orange] (f7) at (1.5,-1.5) {};
 
 \node[circle, minimum size=1mm, inner sep=2, fill=black] at (1,0) {};
 \node[circle, minimum size=1mm, inner sep=2, fill=black] at (0,1) {};
 \node[circle, minimum size=1mm, inner sep=2, fill=black] at (-1,0) {};
 \node[circle, minimum size=1mm, inner sep=2, fill=black] at (0,-1) {};
 \node at (0.9, 0.3) {$v^*$};
 \node at (0.3, 1) {$v_1$};
 \node at (-0.8, 0.3) {$v_2$};
 \node at (0.3, -1) {$v_3$};
 
 \draw[dashed, gray, thick] (F) -- (f7) -- (f1) -- (f2) -- (f3) -- (F) -- (f4) -- (f5) -- (f6) -- (f7);
 \draw[dashed, gray, thick] (f5) -- (F) -- (f6);
 \draw[dashed, gray, thick] (f3) -- (f4);

 \end{scope}
 \begin{scope}[shift={(10,0)}]
 	
 \filldraw[draw=cyan, very thick, fill=cyan!10]	(0,0) circle (1cm);
 \filldraw[draw=red, very thick, fill=red!10]	(60:2) circle (1cm);
 \filldraw[draw=red, very thick, fill=red!10]	(120:2) circle (1cm);
 \filldraw[draw=RoyalBlue, very thick, fill=RoyalBlue!10]	(180:2) circle (1cm);
 \filldraw[draw=ForestGreen, very thick, fill=ForestGreen!10]	(240:2) circle (1cm);
 \filldraw[draw=ForestGreen, very thick, fill=ForestGreen!10]	(300:2) circle (1cm);
 
 \node[circle, minimum size=2mm, fill=cyan] (F) at (0,0) {};
 \node[circle, minimum size=2mm, fill=red] (f1) at (60:2) {};
 \node[circle, minimum size=2mm, fill=red] (f2) at (120:2) {};
 \node[circle, minimum size=2mm, fill=RoyalBlue] (f3) at (180:2) {};
 \node[circle, minimum size=2mm, fill=ForestGreen] (f4) at (240:2) {};
 \node[circle, minimum size=2mm, fill=ForestGreen] (f5) at (300:2) {};
 
 \draw[dashed, gray, thick] (f5) -- (F) -- (f4) -- (f3) -- (f2) -- (F) -- (f1);
 \draw[dashed, ultra thick, ForestGreen] (f4) -- (f5);
 \draw[dashed, ultra thick, blue] (f3) -- (F);
 \draw[dashed, ultra thick, red] (f1) -- (f2);
 \node[circle, minimum size=1mm, inner sep=2, fill=black] at (0,-1.732) {};
 \node[circle, minimum size=1mm, inner sep=2, fill=black] at (0,1.732) {};
 \node[circle, minimum size=1mm, inner sep=2, fill=black] at (-1,0) {};
 \end{scope}
\end{tikzpicture}
	\bigskip
	\caption{Example for two possible situations in Case 2 of the proof of \Cref{lem:clouds}.
	The colored faces are the faces in $\cN$, the light blue face in the middle is $F$.
	The dashed lines represent the reduced conflict graph $R$.\\
	On the left hand side, we see the case where $\cB\neq \emptyset$:
	$\cB$ consists of the red faces that clearly contain $v^*$.
	$F$ and the orange faces also contain $v^*$.
	The only 3 faces that do not contain $v^*$ are the green faces, but they contain $v_1$, $v_2$, and $v_3$, respectively, so $\{v_1, v_2, v_3, v^*\}$ is an $\cF'$-transversal for any $\cF^\prime$.\\
	On the right hand side, we see the case where $\cB = \emptyset$ and $|\cF^\prime| = |\cF| = 6$.
	$R$ contains a perfect matching on $\cF$ as indicated by the thick, colored lines.
	This induces an $\cF^\prime$-transversal of size $3$, indicated by the black vertices.
	}
	\label{fig:clouds}
\end{figure}
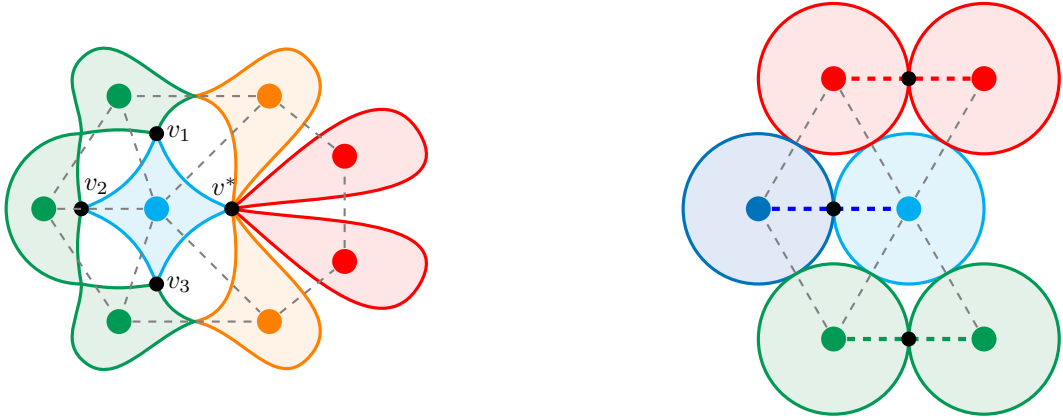

\begin{proof}
    Let $R$ be the reduced conflict graph for $\cW$.
    As $R$ is planar and has no homotopic edges, there are two cases by \Cref{lem:min_degree}:
    \begin{description}
        \item[Case 1:] There is $F\in V(R)$ with $|\delta_R(F)| \leq 4$.
        Let $\cN \subseteq \cW$ consist of all faces of $\cW$ that have at least one common vertex with $F$.
        
        If $\nu(\cN) > 1$, then $\cF:= \cN$ has the demanded properties:
        Items~\ref{item:connected}, \ref{item:packingnumber} and \ref{item:centralface} hold by construction.
        Moreover, as $|\delta_R(F)| \leq 4$, we can find a set $T \subseteq V(F)$ with $|T| \leq 4$ such that $T \cap V(F') \neq \emptyset$ for all $F'\in \cF$ by \cref{lem:degree_in_reduced_conflict_graph}.
        Thus, the graph $\vf(G)[\cF\cup T]$ is connected.
        In particular, $T$ is an $\cF'$-transversal for any $\cF'\subseteq \cF$ where $|\cF'|$ is even and therefore $\tau(\cF^\prime) \leq 4$.
        
        If $\cN = \cW$ then $\nu(\cN) = \nu(\cW) > 1$ and we are in the above case with which we already dealt.
        
        So let us now consider the case where $\nu(\cN) = 1$ and thus in particular $\cN \neq \cW$.
        We take a face $F^*\in \cW \setminus \cN$ that has a common vertex with some $F^\prime \in \cN$ (say $v'\in V(F')\cap V(F^*)$). Define $\cF := \cN \cup \{F^*\}$. Note that $\nu(\cF) > 1$ because $V(F) \cap V(F^*) = \emptyset$.
        Furthermore, \ref{item:connected} and \ref{item:centralface} hold by construction. It remains to prove \ref{item:transversalnumber}:
        
        Let $\cF' \subseteq \cF$ have even cardinality. If $F^*\notin \cF'$, we have $\cF^\prime \subseteq \cN$. Since $\nu(\cN) = 1$, we have $\tau(\cF')\leq 2$ by \cref{lem:packing_number1}.
        
        If $F^* \in \cF'$, consider $\cF_1 := \cF' \oplus \{F^*, F'\}$ and $\cF_2 := \{F^*, F'\}$. Since $F^*\notin \cF_1$, there is an $\cF_1$-transversal $T_1\subseteq V(G)$ with $|T_1|\leq 2$ (again by \cref{lem:packing_number1}). Moreover $\{v^\prime\}$ is an $\cF_2$-transversal. Hence $T_1 \cup \{v^\prime\}$ is an $\cF_1\oplus \cF_2$-transversal by \cref{prop:symmetric_difference_of_transversals}. But $\cF_1\oplus \cF_2= \cF'\oplus \{F^*, F'\} \oplus \{F^*, F'\} = \cF'$ and $|T_1 \cup \{v^\prime\}| \leq 3$.

        \item[Case 2:] There is $F\in V(R)$ with $|\delta_R(F)| = 5$ and $F$ has at least four incident faces in $R$ with degree $3$.
        As in the first case, we consider the set $\cN \subseteq \cW$ of faces that have at least one common vertex with $F$.
        If $\nu(\cN) = 1$, we can proceed in exactly the same way as in the situation in Case 1 where $\nu(\cN) = 1$ and thus in particular $\cN \neq \cW$.
        
        \medskip
        
        Now suppose that we have $\nu (\cN) > 1$. We show that then $\cF:= \cN$ has the demanded properties: Items~\ref{item:connected}, \ref{item:packingnumber}, and \ref{item:centralface} already hold by construction.
        
        Let $\cB$ be the set of faces $F' \in \cF\setminus\{F\}$ for which there is no edge $\{F, F'\}$ in $E(R)$.
        First consider the case that $\cB\neq \emptyset$.
        In $R$, each $F'\in \cB$ is incident to a face $B$ of $R$ with $F \in V(B)$ and $\deg(B) \geq 4$.
        Since there is at most one face with degree $\geq 4$ in $R$ that is incident to $F$, all $F'\in \cB$ are incident to the same face $B$ of $R$.
        
        Choose $v^* \in V(G)$ such that the boundary edges of $B$ are the edges added for $v^*$ in Definition~\ref{def:reduced_conflict_graph}.
        In particular, all faces of $\cB$ contain $v^*$. Furthermore, at least three faces of $\cF \setminus \cB$ (including $F$) contain $v^*$. See also \Cref{fig:clouds}.
        
        Since $|\cF \setminus \cB| \leq 6$, there are at most $3$ faces in $\cF$ that do not contain $v^*$.
        Hence we can find $v_1, v_2, v_3 \in V(F)$ such that $V(F') \cap \{v_1, v_2,v_3, v^*\} \neq \emptyset$ for each $F' \in \cF$. In particular, $\vf(G)[\cF\cup \{v_1,v_2,v_3, v^*\}]$ is connected and thus $\tau(\cF') \leq 4$ for any subset $\cF'\subseteq \cF$ of even cardinality.
        
        If however $\cB = \emptyset$, there is an edge $\{F, F'\} \in E(R)$ for all faces $F'\in \cF\setminus \{F\}$.
        Then  $|\cF| \leq 6$.
        We distinguish again two cases for even subsets $\cF'\subseteq \cF$:
        \begin{enumerate}
            \item $|\cF'| \leq 4$.
            For each $F^\prime \in \cF^\prime$, choose $v_{F^\prime} \in V(F) \cap V(F^\prime)$.
            Then $\{v_{F^\prime} : F^\prime \in \cF^\prime \}$ defines an $\cF^\prime$-transversal and $\tau(\cF^\prime)\leq |\cF^\prime| \leq 4$.
            \item $|\cF'| = 6$. In particular we have $\cF' = \cF$. Since $F$ has at least 4 incident faces in $R$ with degree 3, $R$ contains a perfect matching $M$ for $\cF$ (cf.\@~\Cref{fig:clouds}).
            Adding a vertex in $V(F_1) \cap V(F_2)$ for each $\{F_1, F_2\} \in M$ yields an $\cF$-transversal of size $3$. \qedhere
        \end{enumerate}
    \end{description}
\end{proof}

\section{Acknowledgement}

The authors want to thank Dan Kr\'al\textquoteright{} for helpful discussions.

\bibliographystyle{plain}
\bibliography{references.bib}

\end{document}